\documentclass[a4paper,leqno,twoside]{amsart}

\usepackage[english]{babel} 
\usepackage{inputenc, amsmath, amssymb, latexsym,
  epic, epsfig, rotating, fancyheadings, amsthm, pifont, paralist}

\newcommand{\ld}{\ensuremath{,\ldots,}}
\newcommand{\ssq}{\ensuremath{\subseteq}}
\newcommand{\smin}{\ensuremath{\setminus}}

\newcommand{\N}{\ensuremath{\mathbb{N}}} 
\newcommand{\R}{\ensuremath{\mathbb{R}}}
\newcommand{\Z}{\ensuremath{\mathbb{Z}}}

\newcommand{\C}{\ensuremath{\mathbb{C}}}
\newcommand{\E}{\ensuremath{\mathbb{E}}}

\newcommand{\Leb}{\ensuremath{\mathrm{Leb}}}

\newcommand{\inte}{\ensuremath{\mathrm{int}}}

\newcommand{\Seins}{\ensuremath{\mathbb{S}^{1}}}

\newcommand{\alphlist}{\begin{list}{(\alph{enumi})}{\usecounter{enumi}\setlength{\parsep}{2pt}
      \setlength{\itemsep}{1pt} \setlength{\topsep}{5pt}
      \setlength{\partopsep}{3pt}}}

\newcommand{\arablist}{\begin{list}{(\arabic{enumi})}{\usecounter{enumi}\setlength{\parsep}{2pt}
          \setlength{\itemsep}{1pt} \setlength{\topsep}{5pt}
          \setlength{\partopsep}{3pt}}}

\newcommand{\romanlist}{\begin{list}{(\roman{enumi})}{\usecounter{enumi}\setlength{\parsep}{2pt}
              \setlength{\itemsep}{1pt} \setlength{\topsep}{5pt}
              \setlength{\partopsep}{3pt}}}

 \newcommand{\listend}{\end{list}}

\newcommand{\bulletlist}{\begin{list}{$\bullet$}{\setlength{\parsep}{2pt}
                \setlength{\itemsep}{1pt} \setlength{\topsep}{5pt}
                \setlength{\partopsep}{3pt}\setlength{\leftmargin}{15pt}}}

\newcommand{\foot}{\footnote}

\newcommand{\nfolge}[1]{\ensuremath{(#1)_{n\in\mathbb{N}}}}
\newcommand{\kfolge}[1]{\ensuremath{(#1)_{k\in\mathbb{N}}}}

\newcommand{\jfolge}[1]{\ensuremath{(#1)_{j\in\mathbb{N}}}}

\newcommand{\ncup}{\ensuremath{\bigcup_{n\in\N}}}

\newcommand{\nLim}{\ensuremath{\lim_{n\rightarrow\infty}}}

\newcommand{\jLim}{\ensuremath{\lim_{j\rightarrow\infty}}}

\newcommand{\inergsum}{\ensuremath{\sum_{i=0}^{n-1}}}

\newcommand{\ktel}{\ensuremath{\frac{1}{k}}}

\newcommand{\ntel}{\ensuremath{\frac{1}{n}}}

\newcommand{\Proj}{\ensuremath{\mathbb{P}}}

\newtheoremstyle{tobthm}{3pt}{3pt}{\itshape}{0pt}{\bfseries}{.}{0.5eM}{}
\theoremstyle{tobthm}
\newtheorem{definition}{Definition}[section]
\newtheorem{thm}[definition]{Theorem}

\newtheorem{lem}[definition]{Lemma}
\newtheorem{lemma}[definition]{Lemma}
\newtheorem{cor}[definition]{Corollary}
\newtheorem{prop}[definition]{Proposition}

\newtheoremstyle{tobrem}{3pt}{3pt}{\normalfont}{0pt}{\bfseries}{.}{0.5em}{}
\theoremstyle{tobrem} 
\newtheorem{rem}[definition]{Remark} \newtheorem{example}[definition]{Example}

\newtheorem{problem}[definition]{Problem}

\numberwithin{equation}{section}
\numberwithin{figure}{section}

\title{\Large\textsc{Random minimality and continuity of invariant graphs in
    random dynamical systems}} \author{T.~J\"ager \and G.~Keller}

\newcommand{\cB}{\mathcal{B}}

\newcommand{\cF}{\mathcal{F}}

\newcommand{\cM}{\mathcal{M}}
\newcommand{\cU}{\mathcal{U}}

\newcommand{\cK}{\mathcal{K}}
\newcommand{\cL}{\mathcal{L}}
\newcommand{\cI}{\mathcal{I}}

\newcommand{\cX}{\mathcal{X}}

\newcommand{\Kmax}{K}

\newcommand{\Romanlist}{\begin{list}{(\Roman{enumi})}{\usecounter{enumi}\setlength{\parsep}{2pt}
              \setlength{\itemsep}{1pt} \setlength{\topsep}{5pt}
              \setlength{\partopsep}{3pt}}}

\newcommand{\CC}{{\widehat C}}

\renewcommand{\geq}{\geqslant}
\renewcommand{\leq}{\leqslant}

\begin{document}

\setlength{\abovedisplayskip}{0.8ex}
\setlength{\abovedisplayshortskip}{0.6ex}

\setlength{\belowdisplayskip}{0.8ex}
\setlength{\belowdisplayshortskip}{0.6ex}

\begin{abstract}
  We study dynamical systems forced by a combination of random and deterministic
  noise and provide criteria, in terms of Lyapunov exponents, for the existence
  of random attractors with continuous structure in the fibres.  For this
  purpose, we provide suitable random versions of the semiuniform ergodic
  theorem and also introduce and discuss some basic concepts of random
  topological dynamics.\bigskip

  \noindent {\em 2010 Mathematics Subject Classification:} 37A30,
  37H15, 34D45.
\end{abstract}

\maketitle

\section{Introduction}

Our aim is to introduce a mathematical framework for the investigation
of dynamical systems forced by a combination of both deterministic and
random external factors, and to discuss conceptual issues that arise in
this context. In general, dynamical systems under the influence of
external forcing are modeled, in discrete time, as skew products
\begin{equation}
  \label{eq:skewproduct}
  T : \Omega\times M \to \Omega \times M \quad , 
  \quad T(\omega,x)=(\theta\omega,T_\omega(x)) \ ,
\end{equation}
where the dynamics of the forcing process are described by the {\em
  base transformation} \mbox{$\theta:\Omega\to\Omega$.} Note that for
simplicity we write $\theta\omega$ instead of $\theta(\omega)$.
Typically, depending on the nature of the forcing, $\theta$ is assumed
to be either a continuous map of a compact metric space $\Omega$ ({\em
  deterministic forcing}) or a measure-preserving transformation of a
probability space $(\Omega,\mathcal{F},\Proj)$ ({\em random forcing}).
In a similar way, modeling in continuous time leads to skew product
flows, which again give rise to skew product systems of the form
(\ref{eq:skewproduct}) via their time-one maps.

An {\em invariant graph} of $T$ is the graph of a measurable function
$\varphi:\Omega\to M$ which satisfies
\begin{equation}
  \label{eq:invgraph}
  T_\omega(\varphi(\omega)) \ = \ \varphi(\theta\omega) 
\end{equation}
for all (or $\Proj$-almost all) $\omega\in\Omega$. More generally, invariant
graphs may also be {\em multivalued}, that is, consist of a constant number
$n\geq 2$ of points in each fibre. In the study of forced or non-autonomous
dynamical systems of the above form, invariant graphs play a central role since
they are the natural substitutes of fixed points of autonomous systems. Lyapunov
exponents yield additional information about the stability and attractivity of
invariant graphs. When $M$ is a smooth manifold and the fibre maps
$T_\omega$ are all differentiable, the maximal Lyapunov exponent of $\varphi$
with respect to a $\theta$-invariant probablity measure $\Proj$ is defined as
\begin{equation}
  \label{eq:lyap}
  \lambda_m(\varphi) \ = \ \nLim \ntel \int_\Omega \log 
  \| DT^n_\omega(\varphi(\omega))\| \ d\Proj(\omega) \ .
\end{equation}
Here $T^n_\omega = T_{\theta^{n-1}\omega}\circ \cdots \circ T_\omega$
and $DT^n_\omega(x)$ denotes the derivative of $T^n_\omega$ in $x$.
Note that the limit in \eqref{eq:lyap} exists by Kingmans Ergodic
Theorem.  For the interpretation of the Lyapunov exponents in
deterministically forced systems, the continuity of the invariant
graph is crucial. In particular, continuous invariant graphs with only
negative Lypunov exponents are known to be uniformly attracting and
hence stable under perturbation, which is not true in the
non-continuous case. A context in which the attractivity of invariant
graphs plays a central role is {\em generalised synchronisation}, a
phenomenon that has been widely studied in theoretical physics
\cite{rulkov-et-al1995,pyragas1996,pikovsky/rosenblum/kurths,keller/jafri/ramaswamy2013}
and, more recently, also in mathematics \cite{Homburg2013Synchronization}.
The transition from continuous to non-continuous invariant graphs is
of interest in the context of non-autonomous bifurcation theory and
has equally been studied intensively in different forms, such as {\em
  fractalisation} and {\em torus collision}
\cite{prasad/negi/ramaswamy:2001}.

A criterion ensuring the existence of continuous invariant graphs has
been provided by Sturman and Stark in \cite{sturman/stark:2000}. In
order to state it, we define the {\em maximal Lyapunov exponent} of a
$T$-invariant probability measure $\mu$ as
\begin{equation}
  \label{eq:maxlyap}
  \lambda(\mu,T) \ = \ \nLim \ntel \int_\Omega \log \|DT^n_\omega(x)\| \ d\mu(\omega,x) \ .
\end{equation}
Note that (\ref{eq:lyap}) is a special case of (\ref{eq:maxlyap}),
with $\mu$ given by $\mu(A)=m(\{\omega\in\Omega\mid
(\omega,\varphi(\omega))\in A\}$.
\begin{thm}[\cite{sturman/stark:2000}, Theorem ]
  \label{t.sturmanstark} Suppose $\theta$ is a minimal homeomorphism
  of a compact metric space $\Omega$ and $f$ is a skew product map of
  the form (\ref{eq:skewproduct}) with $M=\R^d$ such that the fibre
  maps $T_\omega$ are differentiable and their derivative depends
  continuously on $(\omega,x)$. Further, assume that $K$ is a compact
  $T$-invariant set and $\lambda(\mu,T)<0$ for all $T$-invariant
  probability measures $\mu$ supported on $K$. Then $K$ is a finite
  union of continuous invariant graphs (possibly multi-valued).
\end{thm}
In fact, the above statement is a slight generalisation of the one in
\cite{sturman/stark:2000} and taken from
\cite{AnagnostopoulouJaeger2012SaddleNodes}, since Sturman and Stark restricted
to skew products over irrational translations on a $d$-dimensional torus. In
this case, further conclusions can be drawn concerning the regularity of the
invariant graphs, which turn out to be as smooth as the system itself, see
\cite{stark:1999}.\smallskip

In the context of random forcing, one can a priori not speak about the continuity
of invariant graphs, due to the lack of a topological structure on the driving
space $\Omega$. However, a first question one may ask is whether under similar
assumptions a random compact set is just a finite union of invariant graphs, and
thus consists of a finite number of points on each fibre. In order to state the
respective analogue to Theorem~\ref{t.sturmanstark}, we need to introduce some
terminology concerning random dynamical systems.  

A {\em random map with base } $(\Omega,\mathcal{F},\Proj,\theta)$, in
the sense of Arnold \cite{Arnold1998RandomDynamicalSystems}, is a skew
product of the form (\ref{eq:skewproduct}) where
$(\Omega,\mathcal{F},\Proj)$ is a probability space,
$\theta:\Omega\to\Omega$ is a bi-measurable and ergodic
measure-preserving bijection and $M$ is a measurable space.\foot{The
  base $(\Omega,\mathcal{F},\Proj,\theta)$ is called a {\em
    measure-preserving dynamical system}.}  When $M$ is metric, we
always assume that the measurable structure is given by the Borel
$\sigma$-algebra or its completion.  Further, it is required that the
fibre maps $T_\omega:M\to M$ are continuous and that for each $x\in M$
the map $\omega\mapsto T(\omega,x)$ is $\cF$-measurable.  If $M$ is a
smooth manifold and all fibre maps $T_\omega$ are $\mathcal{C}^r$, we
call $T$ a random $\mathcal{C}^r$-map. If the fibre maps are all
homeomorphisms, we call $T$ a {\em random homeomorphism}.  The set of all
$T$-invariant probability measures on $\Omega\times M$ is denoted by
by $\cM(T)$, and the set of $\mu\in\cM(T)$ which project to $\Proj$ by
$\cM_\Proj(T)$. Given $K\ssq \Omega\times M$, we let $K(\omega)=\{x\in
M\mid (\omega,x)\in K\}$. We say $K$ is a {\em random compact (closed)
  set} if
\begin{enumerate}[(i)]
\item $K(\omega)=\{x\in M\mid (\omega,x)\in K\}$ is compact (closed) for all
  $\omega\in\Omega$;
\item the functions $\omega\mapsto d(x,K(\omega))$ are measurable for
  all $x\in M$.
\end{enumerate}
We now have 
\begin{thm}
  \label{t.singlepoints} Suppose $T:\Omega\times\R^d\to\Omega\times\R^d$ is a
  random $\mathcal{C}^1$-map with base $(\Omega,\cF,\Proj,\theta)$, the family
  $\left(x\mapsto\log \|DT^k_\omega(x)\|\right)_{\omega\in\Omega}$ is
  equicontinuous for all $k\in\N$ and $K\ssq \Omega\times \R^d$ is a random
  compact set such that $\lambda(\mu,T)<0$ for all $\mu\in\cM_\Proj^K(T)$. Then
  there exists an integer $n$ such that $\# K(\omega)= n$ for $\Proj$-a.e.\
  $\omega\in\Omega$.
\end{thm}
This theorem is a special case of Corollary~\ref{c.singlepoints} in
Section~\ref{Applications} below. \smallskip

Now, in the situation where the forcing has a deterministic component
it does make sense to ask for the continuous dependence of
invariant graphs on the deterministic variables, even if the graph as
a whole is random. In order to make this more precise, we consider
systems forced by a combination of random and deterministic noise and
model them as double skew products of the form
\begin{equation}
   \label{e.10intro}
  T :\Omega\times\Xi\times \R^d \to\Omega\times\Xi\times \R^d \quad , 
 \quad T(\omega,\xi,y) = (\theta\omega,g_\omega(\xi),h_{\omega,\xi}(y)) \ ,
\end{equation}
where $(\Omega,\cF,\Proj,\theta)$ is a measure-preserving dynamical system and
$\Xi$ is a compact metric space. The second component $\Xi$ of the product space
corresponds to the deterministic part of the forcing, which we allow to depend
on the random noise as well. Note that fibre maps of $T$ acting on $M=\Xi\times
\R^d$ are given by $T_\omega(\xi,y)=(g_\omega(\xi),h_{\omega,\xi}(y))$.

To our knowledge, such systems have not been studied in the literature
before. However, from the point of view of modeling real-world
processes, such an overlay of random and deterministic forcing seems
very natural. Furthermore, models of this type come up as well in the
bifurcation theory of purely random dynamical systems
\cite{anagnostopoulou/jaeger/keller:2012}, which provided the starting
point of the work presented here.

When $(\Omega\times\Xi,\theta\ltimes g)$ is considered as the basis of
(\ref{e.10intro}), then an invariant graph is a measurable function
$\varphi:\Omega\times\Xi \to \R^d$ which satisfies
\begin{equation}
  \label{eq:invgraph2}
  h_{\omega,\xi}(\varphi(\omega,\xi)) \ = \ \varphi(\theta\omega,g_\omega(\xi)) \ .
\end{equation}
Again $\varphi$ cannot be continuous in $\omega$, but $\xi\mapsto
\varphi(\omega,\xi)$ may be continuous for $\Proj$-a.e.\ $\omega\in\Omega$, in
which case we call $\varphi$ a {\em random continuous invariant graph}. In order
to obtain a random analogue to Theorem~\ref{t.sturmanstark}, we have to assume
the random map 
\begin{equation} \label{e.thetag} \theta\ltimes g : \Omega\times
  \Xi\to\Omega\times\Xi\quad , \quad (\omega,\xi)\mapsto
  (\theta\omega,g_\omega(\xi))
\end{equation}
to be a random minimal homeomorphism, which we define as follows.
\begin{definition}\label{def:minimality}
  The random homeomorphism $\theta\ltimes g$ is \emph{minimal}, if
  each $(\theta\ltimes g)$-forward invariant random closed set $K$
  obeys the following dichotomy:
\begin{description}
\item[either\,] $K(\omega)=\Xi$ for $\Proj$-a.e. $\omega$,
\item[or{\phantom{eith}}\,] $K(\omega)=\emptyset$ for $\Proj$-a.e. $\omega$.
\end{description}
\end{definition}
Finally, by $\cK(\R^d)$ we denote the set of all compact subsets of $\R^d$ and
equip it with the Hausdorff distance to make it a metric space. Using these
notions, we obtain
\begin{thm}\label{t.random-sturman-stark_intro}
  Let $T$ be a random map of the form (\ref{e.10intro}) and $K$ be a
  $T$-invariant random compact set and suppose that for all $k\in\N$
  the family $\left((\xi,y)\mapsto
    \log\|Dh^k_{\omega,\xi}(y)\|\right)_{\omega\in\Omega}$ is
  equicontinuous.  Then, if $\lambda(\mu,T)<0$ for all measures
  $\mu\in\cM_\Proj^K(T)$, and if $\theta\ltimes g$ is a random minimal
  homeomorphism on $\Omega\times\Xi$, there exists a (non-random)
  integer $n>0$ such that, for $\Proj$-a.e. $\omega\in\Omega$,
  \romanlist
\item $\# K(\omega,\xi)=n$ for all $\xi\in\Xi$, 
\item the map $\xi\mapsto K(\omega,\xi)$ from $\Xi$ to $\cK(\R^d)$ is
  continuous.  \listend
\end{thm}
Note that (ii) implies that $K$ can be represented as a finite union
of random continuous invariant graphs.  We restate and prove this
statement, in slightly more general form, as
Theorem~\ref{t.random-sturman-stark} in Section~\ref{Applications}. As
mentioned in Remark~\ref{r.random-sturman-stark}, a slightly weaker
result still holds when $g$ is only random transitive, a notion which
is also introduced in Section~\ref{Applications}. If $K$ is
connected in each fibre, then the minimality assumption on $g$ can
even be dropped completely (Theorem~\ref{t.connected}).\smallskip

A crucial ingredient in the proof of Theorems~\ref{t.singlepoints} and
\ref{t.random-sturman-stark_intro} is a random version of the
semi-uniform ergodic theorem
\cite{Schreiber1998SET,sturman/stark:2000}.  Such a result was already
proved by Cao \cite{Cao2006RandomSET}, but since we need a non-trivial
modification of his statement (our Theorem~\ref{t.complement-to-main})
we provide this in Section~\ref{RSET} with an independent proof.

From a conceptual point of view, an interesting aspect of these
studies is the fact that notions of random topological dynamics, like
random minimality defined above and random transitivity introduced in
Section~\ref{RandomTopDyn}, come up naturally in our context.  Despite
the importance of these concepts in autonomous dynamics, it seems that
random analogues have not been considered before.  Some basic facts
are collected in Section~\ref{RandomTopDyn}, although these can merely
serve as starting points for the development of a more comprehensive
theory of random topological dynamics.\bigskip

\noindent \textbf{Acknowledgements.} The authors were supported by the
German Research Council (TJ by Emmy-Noether-Project Ja 1721/2-1, GK by
DFG-grant Ke 514/8-1).  Further, this work is part of the activities
of the Scientific Network ``Skew product dynamics and multifractal
analysis'' (DFG-grant Oe 538/3-1).

\section{The random semiuniform ergodic theorem}
\label{RSET}

Recall that a sequence of measurable functions $\Phi_n:M\to\R$ is subadditive
with respect to a measurable transformation $T:M\to M$ if
\begin{equation}
  \label{e.2}
\Phi_{n+m} \ \leq \ \Phi_n\circ T^m + \Phi_m \quad \textrm{for all } n,m\in\N \ .
\end{equation}
Let $\mathcal{M}(T)$ denote the set of $T$-invariant Borel probability
measures on $M$. Given $\mu\in\mathcal{M}(T)$, \eqref{e.2} implies
$\mu(\Phi_{n+m})\leq \mu(\Phi_n) + \mu(\Phi_m)$ provided both sides
are well defined, such that by subadditivity the limit 
\begin{displaymath}
\overline{\Phi}_\mu:=\nLim\ntel\mu(\Phi_n) =\inf_{n\in\N} \ntel\mu(\Phi_n)
\end{displaymath}
exists.  Kingman's Subadditive Ergodic Theorem also ensures the
$\mu$-a.s.\ existence of the pointwise limit
\begin{equation}
  \label{e.4}
  \bar\Phi(x) \ = \ \nLim \ntel \Phi_n(x) \ ,
\end{equation}
provided $\Phi_1$ is integrable. Further
$\mu(\bar\Phi)=\overline\Phi_\mu$, and if $\mu$ is ergodic we also
have $ \bar\Phi(x) = \overline\Phi_\mu $ for $\mu$-a.e.~$x\in M$.

When $M$ is a compact metric space, $\Phi$ is continuous and $T$ is a continuous
and {\em uniquely ergodic} map, the latter meaning that there exists a unique
$T$-invariant Borel probability measure on $M$, then this statement can be
strengthened by replacing pointwise with semi-uniform convergence on $M$ in
(\ref{e.4}). In fact, for many applications, instead of unique ergodicity it
suffices to have an upper bound for the limits $\overline{\Phi}_\mu$ with
respect to all $T$-invariant probability measures $\mu$. In particular, this
applies when uniform contraction estimates are derived from negative Lyapunov
exponents. We have
\begin{thm}[Semiuniform Ergodic Theorem,
  \cite{Schreiber1998SET,sturman/stark:2000}]
  \label{t.SET}
  Let $M$ be a compact metric space, $T:M\to M$ a continuous
  transformation and $\nfolge{\Phi_n}$ a subadditive sequence of
  continuous functions. Suppose that $\lambda\in\R$ satisfies 
  $\overline\Phi_\mu<\lambda$ for all $\mu\in\mathcal{M}(T)$. Then there
  exist $n_0\in\N$ and $\delta>0$ such that
  \[
  \ntel\Phi_n(x) \ \leq \ \lambda-\delta \quad \textrm{for all } n\geq n_0,\ x\in M.
  \]
\end{thm}

Our aim is to provide a random analogue to this result, which is needed
in the proof of Theorem~\ref{t.random-sturman-stark_intro} in 
Section~\ref{Applications}. As before, we denote the set of all $T$-invariant
probability measures on $\Omega\times M$ by $\cM(T)$, and the set of
$\mu\in\cM(T)$ which project to $\Proj$ by $\cM_\Proj(T)$. We call
$\Phi:\Omega\times M\to\R$ a \emph{random continuous function} if
\begin{enumerate}[(i)]
\item the functions $x\mapsto\Phi(\omega,x)$ are continuous for all
  $\omega\in\Omega$;
\item the functions $\omega\mapsto\Phi(\omega,x)$ are measurable for
  all $x\in M$.
\end{enumerate}
$U\ssq \Omega\times M$ is a {\em random open set} if $U^c$ is random closed.  A
random set $K$ (open or closed) is called {\em forward $T$-invariant} if
$T_\omega(K(\omega))\ssq K(\theta\omega)$ for $\Proj$-almost every
$\omega\in\Omega$.  $K$ is called {\em $T$-invariant} if the inclusion can be
replaced by equality.  Note that in contrast to random $T$-invariance, the
notions of random continuous functions and random open, closed or compact sets do not
depend on the measure $\Proj$.

For any forward $T$-invariant random compact set $K$, we denote the set of
$\mu\in\mathcal{M}_\Proj(T)$ which are supported on $K$ by
$\mathcal{M}^K_\Proj(T)$.  Given a random continuous function $\Phi:\Omega\times
M\to\R$ and a random compact set $K$, we define
\begin{equation}
  \Phi^{\Kmax}(\omega)  :=  \max\{ \Phi(\omega,x) \mid x\in K(\omega) \} \  .
\end{equation}  
Finally, we call a random variable $C:\Omega\to \R$ {\em adjusted} to
$\theta$, if it satisfies $\lim_{|n|\to\infty}\frac{1}{|n|}
C(\theta^n\omega)=0$ for $\Proj$-a.e.\ $\omega$. 
\begin{thm}[Random Semiuniform Ergodic Theorem] \label{t.random_semiuniform} Let
  $T:\Omega\times M\to \Omega\times M$ be a random map with ergodic base
  $(\Omega,\cF,\Proj,\theta)$. Suppose that $K$ is a forward $T$-invariant
  random compact set, and that $\nfolge{\Phi_n}$ is a subadditive sequence of
  random continuous functions with $|\Phi_n|^{\Kmax}\in L^1(\Proj)$ for all
  $n\in\N$. Further, assume that $\lambda\in\R$ satisfies $\overline\Phi_\mu <
  \lambda$ for all $\mu\in\mathcal{M}^K_\Proj(T)$.
  
  Then there exist $\lambda'<\lambda$
  and an adjusted random variable $C:\Omega\to\R$ such that
  \begin{equation} \label{e.main-estimate}
    \Phi_n(\omega,x) \ \leq \ C(\omega)+n\lambda' 
    \quad \textrm{for all $n\in\N$, $\Proj$-a.e.\ $\omega\in\Omega$ and all $x\in K(\omega)$.} 
  \end{equation}   
  In particular, for $\delta\in (0,\lambda-\lambda')$ and $\Proj$-a.e.\
  $\omega\in\Omega$ there exists $n(\omega)\in \N$ such that
  \begin{equation}
    \ntel \Phi_n(\omega,x) \ \leq \ \lambda -\delta \qquad \textrm{for all }  n\geq n(\omega) 
    \textrm{ and } x\in K(\omega).
  \end{equation}
\end{thm}
\begin{rem} 
\begin{enumerate}[(a)]
\item This theorem generalises \cite[Theorem 1]{Schreiber1998SET} and
  \cite[Theorem 1.9]{sturman/stark:2000} from the deterministic to the
  random setting. It is slighty more general than the main result in
  \cite{Cao2006RandomSET}, mostly because of our estimate
  (\ref{e.main-estimate}), but also because our
  Lemma~\ref{lemma:measurability} allows to avoid the assumption that
  the probability space $(\Omega,\cF,\Proj)$ is complete. As estimate
  (\ref{e.main-estimate}) and its more subtle variant given in
  Theorem~\ref{t.complement-to-main} below are crucial for various
  applications, we give a streamlined full proof of both results.
    
     It should be mentioned that \cite{sturman/stark:2000}
    also contains a random version of these results (Theorem 1.19),
    but only in the sense that the statement is made with respect to a
    fixed reference measure on the base, while the space $\Omega$ is
    still assumed to be compact and both $T$ and $\Phi_n$ are required
    to be continuous. This theorem is therefore not applicable to
    general randomly forced systems, and the proof avoids the
    difficulties coming from the lack of a topological structure on
    $\Omega$ that we have to deal with here.
\item Note that $C$ is adjusted in the above sense if and only if
  $e^C$ is tempered \cite[Definition 4.1.1]{Arnold1998RandomDynamicalSystems}. Hence, by
  \cite[Proposition 4.3.3]{Arnold1998RandomDynamicalSystems}, given $\epsilon>0$ there is a
  random variable $D_\epsilon$ such that
\begin{equation*}
  -\epsilon|n|+D_\epsilon(\omega)\leqslant 
  D_\epsilon(\theta^n\omega)\leqslant\epsilon|n|+D_\epsilon(\omega)\text{ for all $n\in\Z$}
\end{equation*}
and $C\leqslant D_\epsilon$, so that 
\begin{equation} \nonumber
  \Phi_n(\omega,x) \ \leq \ D_\epsilon(\omega)+n\lambda' \quad \textrm{for } \Proj\textrm{-a.e. } 
  \omega\in\Omega \textrm{ and all } x\in K(\omega). 
\end{equation}
\item Theorem~\ref{t.random_semiuniform} can easily be extended to
  subadditive sequences of random continuous functions
  $\Phi_n:\Omega\times M\to\R\cup\{-\infty\}$, see
  Remark~\ref{r.minus_infinity}. This is particularly important in the
  light of applications to Lyapunov exponents involving non-invertible
  linear cocycles or differential matrices.
\end{enumerate}
\end{rem}
\medskip

In order to prove Theorem~\ref{t.random_semiuniform}, we start by providing some
more basic facts on random dynamical systems.  The following lemma is easily
derived from results in \cite{CV1977}, but we include the proof for the
convenience of the reader. Note that unlike the related statements \cite[Lemma
III.39]{CV1977} and \cite[Theorem 8.2.11]{AF1990}, it does not require
completeness of the probability space $(\Omega,\cF,\Proj)$.
\begin{lemma}\label{lemma:measurability}
$\Phi^{\Kmax}:\Omega\to\R$ is measurable, and there is a measurable map
$h:\Omega\to M$ such that $h(\omega)\in K(\omega)$ and
$\Phi^{\Kmax}(\omega)=\Phi(\omega,h(\omega))$ for all $\omega\in\Omega$.
\end{lemma}
\begin{proof}
  For the random compact set $K$ there is a sequence $\kfolge{a_k}$ of
  measurable maps $a_k:\Omega\to M$ such that
  $K(\omega)=\operatorname{closure}\{a_k(\omega)\mid k\in\N\}$ for all
  $\omega\in\Omega$ \cite[Theorem III.30]{CV1977}, see also \cite[Proposition
  1.6.3]{Arnold1998RandomDynamicalSystems}. Therefore, observing the random continuity of $\Phi$,
\begin{equation*}
\Phi^{\Kmax}(\omega)  =  \max\{ \Phi(\omega,x) \mid x\in K(\omega) \}
= \sup\{\Phi(\omega,a_k(\omega))\mid k\in\N\}
\end{equation*}
is measurable. 
For $\ell\in\N$ let $k_\ell(\omega)=\min\left\{k\in\N \mid
  \Phi(\omega,a_k(\omega))>\Phi^{\Kmax}(\omega)-\ell^{-1}\right\}$. The
$k_\ell:\Omega\to\N$ are measurable.
 Therefore, 
\begin{equation*}
  H(\omega)=\bigcap_{j\in\N}\overline{\bigcup_{\ell\geqslant j}\{a_{k_\ell(\omega)}(\omega)\}}
  \subseteq
  \left\{x\in K(\omega)\mid \Phi(\omega,x)=\Phi^{\Kmax}(\omega)\right\}
\end{equation*} 
is a random compact set \cite[Proposition III.4]{CV1977}, and there is a
measurable selection $h:\Omega\to M$ such that $h(\omega)\in H(\omega)\subseteq
K(\omega)$ for all $\omega\in\Omega$ \cite[Theorem III.9]{CV1977}.
\end{proof}

The proof of the next lemma is straightforward.
\begin{lemma}
  If the sequence $\nfolge{\Phi_n}$ of random continuous functions is
  subadditive and $K$ is a forward invariant random compact set, then 
  the sequence $\nfolge{\Phi_n^{\Kmax}}$ is subadditive.
\end{lemma}

\label{sec:main}

From now on, we use the hypothesis of
Theorem~\ref{t.random_semiuniform} as standing assumptions for the
remainder of this section.  Subadditivity of \nfolge{\Phi^{\Kmax}_n}
allows to define
\[
\overline\Phi^{\Kmax} \ = \ \inf_{n\in\N} 
\ntel \Proj\left(\Phi^{\Kmax}_n\right) = \ \nLim \ntel \Proj\left(\Phi^{\Kmax}_n\right) \ .
\]
Obviously $\overline\Phi_\mu\leq \overline\Phi^{\Kmax}$ for each
$\mu\in\mathcal{M}^K_\Proj(f)$.  The proofs of the following two
results are inspired by the proofs of Lemmas 3 and 4 in
{\cite{Cao2008a}}.  Note that any measure $\mu\in\mathcal{M}_\Proj(T)$
can be disintegrated into a family of probability measures
$(\mu_\omega)_{\omega\in\Omega}$ on the fibres, in the sense that
$\int_{\Omega\times M} \Phi \ d\mu =\int_\Omega \int_M \Phi(\omega,x)
\ d\mu_\omega(x)\,d\Proj(\omega)$ for all measurable functions
$\Phi:\Omega\times M\to \R$ \cite[Proposition 1.4.3]{Arnold1998RandomDynamicalSystems}.

\begin{lem} \label{l.maximizing_measure} We have
  $\overline\Phi^{\Kmax}=\sup\{\overline\Phi_\mu \mid
  \mu\in\mathcal{M}^K_\Proj(T)\}$, and the supremum is attained by some
  $\mu^*\in\mathcal{M}^K_\Proj(T)$.
\end{lem}
\begin{proof}
  Let $h_n$ be measurable selections such that
  $\Phi_n^{\Kmax}(\omega)=\Phi_n(\omega,h_n(\omega))$, see
  Lemma~\ref{lemma:measurability}. Define measures $\mu_n \in
  \mathcal{M}^K_{\Proj}(T)$ via their fibre measures
\[
\mu_{n,\omega}(\theta) \ = \ \frac{1}{n} \sum_{i = 0}^{n
  - 1} \delta_{T^i_{\theta^{-i}\omega}(h_n (\theta^{-i}\omega))} \ ,
\]
where $\delta_x$ denotes the Dirac measure in a point $x\in M$.  As
$h_n (\omega) \in  K (\omega)$ for each
$\omega\in\Omega$, all measures $\mu_n$ are supported by the forward
invariant random compact set $K$. Hence, by the random
Krylov-Bogulyubov Theorem ({\cite{Crauel2002}} or {\cite[Theorem
  1.6.13]{Arnold1998RandomDynamicalSystems}}), there is a subsequence
$(\mu_{n_l})_{l\in\N}$ converging (random-)weakly to some $\mu^{\ast}
\in \mathcal{M}^{K}_{\Proj} (f)$.
  
Now fix $k \in \N$. Then, for some $t\in\{0\ld k-1\}$ the sequence
$(n_l)_{l\in\N}$ contains a subsequence of the form
$(s_jk+t)_{j\in\N}$. Note that for any sequence $\jfolge{x_j}$ of
measurable functions $x_j:\Omega\to M$ with $x_j(\omega)\in K(\omega)$
for all $\omega\in\Omega,\ j\in\N$, and any sequence of integers
$\jfolge{N_j}$ with $N_j\nearrow\infty$, the fact that
$|\Phi_k|^{\Kmax}\in L^1(\Proj)$ easily implies
  \begin{equation}
    \label{e.integral}
    \jLim \frac{1}{N_j} \int_\Omega \Phi_k(\omega,x_j(\omega)) \ d\Proj(\omega) \ = \ 0 \ .
  \end{equation}
  Using this observation, we obtain
\begin{eqnarray*}
  \mu^*\left(\ktel \Phi_k\right) & = & \jLim \ktel \int_\Omega \Phi_k \ d\mu_{s_jk+t} \\
  & = & \jLim\frac{1}{k(s_jk+t)} \sum_{i=0}^{s_jk+t-1} \int_\Omega \Phi_k 
  \circ T^i(\theta^{-i}\omega,h_{s_jk+t}(\theta^{-i}\omega)) \ d\Proj(\omega)\\
  & = & \jLim\frac{1}{k(s_jk+t)} \sum_{i=0}^{s_jk+t-1} \int_\Omega \Phi_k 
  \circ T^i(\omega,h_{s_jk+t}(\omega)) \ d\Proj(\omega) \\
  & \stackrel{(\ref{e.integral})}{=} & \jLim\frac{1}{k(s_jk+t)}
  \sum_{i=0}^{k-1}\sum_{l=0}^{s_j-2} \int_\Omega \Phi_k 
  \circ T^{kl+i}(\omega,h_{s_jk+t}(\omega)) \ d\Proj(\omega) \\
  & \geq & \jLim\frac{1}{k(s_jk+t)} \sum_{i=0}^{k-1} \int_\Omega \Phi_{(s_j-1)k }\circ 
  T^i(\omega,h_{s_jk+t}(\omega)) \ d\Proj(\omega) \\
  & \geq & \jLim\frac{1}{k(s_jk+t)} \sum_{i=0}^{k-1} \left(\int_\Omega 
    \Phi_{s_jk+t}(\omega,h_{s_jk+t}(\omega))\ d\Proj(\omega) \right. \\
  & &  - \ \left.\int_\Omega \Phi_i(\omega,h_{s_jk+t}(\omega)) + 
    \Phi_{k-i+t}\circ T^{(s_j-1)k+i}(\omega,h_{s_jk+t}(\omega))\ d\Proj(\omega)\right)
  \\
  & \stackrel{(\ref{e.integral})}{=} & \jLim\frac{1}{s_jk+t} \int_\Omega 
  \Phi_{s_jk+t}(\omega,h_{s_jk+t}(\omega))\ d\Proj(\omega)  \\
  & = & \jLim \frac{1}{s_jk+t} \int_\Omega \Phi^{\Kmax}_{s_jk+t}(\omega) \ d\Proj(\omega) \quad
  = \quad \overline\Phi^{\Kmax} \ .
\end{eqnarray*}
Since this holds for all $k\in\N$, we have
$\overline\Phi_{\mu^*}=\inf_{k\in\N} \mu^*\left(\ktel \Phi_k\right) \geq
\overline\Phi^{\Kmax}$.
\end{proof}

For the proof of Theorem~\ref{t.random_semiuniform}, we will further
need the following useful criterion for the adjustedness of random
variables.
\begin{lem} \label{l.adjusted} Suppose $C:\Omega\to \R$ is measurable
  and $C\circ \theta-C$ has a $\Proj$-integrable minorant. Then $C$ is
  adjusted to $\theta$.
\end{lem}
\begin{proof} Since $C\circ \theta-C$ has an integrable minorant, \cite[Lemma
4.1.13]{Keller1997b} implies that $C\circ \theta -C \in L^1(m)$ and
$\int_\Omega C\circ \theta-C \ dm=0$. Hence, we obtain from the
Birkhoff Ergodic Theorem that 
\begin{eqnarray*}
  \nLim \ntel C(\theta^n\omega) & = & \nLim \ntel(C(\theta^n\omega)-C(\omega))
  \ = \  \nLim \ntel\inergsum C(\theta^{i+1}\omega)-C(\theta^i\omega) \\
  & = & \nLim \ntel\inergsum (C\circ \theta-C)\circ \theta^i(\omega) \ = \ 0 \
\end{eqnarray*}
for $m$-a.e.\ $\omega\in\Omega$. Since $C\circ\theta^{-1}-C \ = \
-(C\circ\theta-C)\circ\theta^{-1}\in L^1(\Proj)$ as well, the limit for
$n\to-\infty$ can be treated in the same way.  
\end{proof}

\begin{proof}[\bf Proof of Theorem~\ref{t.random_semiuniform}] Suppose
$\lambda>\overline\Phi_\mu$ for all $\mu\in\mathcal{M}^K_\Proj(T)$.  Then
$\lambda>\overline\Phi^{\Kmax}$ by Lemma~\ref{l.maximizing_measure},
and we can choose $\lambda'\in (\overline\Phi^{\Kmax},\lambda)$. We
have $\nLim \ntel \Phi^{\Kmax}_n(\omega) = \overline\Phi^{\Kmax}$ $\Proj$-a.s.\ by
Kingman's Subadditive Ergodic Theorem. If we let $\Phi_0^{\Kmax}=0$,
then the random variable $C$ defined by
\[
C(\omega) \ = \ \sup_{n\geq 0}( -\lambda' n + \Phi^{\Kmax}_n(\omega))
\]
is non-negative and $\Proj$-a.s.\ finite. The fact that $C$ satisfies
(\ref{e.main-estimate}) is obvious from its definition.  Further, if
$C(\omega)=0$ then $C(\theta\omega)-C(\omega)\geq 0$. Otherwise, since 
\begin{eqnarray*}
  - \lambda' n + \Phi_n^{\Kmax} (\omega) &\leq &\left(-\lambda'(n-1)+
    \Phi_{n - 1}^{\Kmax} (\theta\omega)\right) + \left( - \lambda' + \Phi_1^{\Kmax}
    (\omega) \right) \\ &\leq & C (\theta\omega) - \lambda' + \Phi_1^{\Kmax}
  (\omega), 
\end{eqnarray*}
for all $n\geqslant 1$, we have that 
\[ C (\omega) \ \leq \ C(\theta\omega) - \lambda' + \Phi_1^{\Kmax}(\omega) \ .
\]
Combining both estimates yields
\[ C (\theta\omega) - C(\omega) \ \geq \ \min \{0, \lambda' -
\Phi_1^{\Kmax}(\omega)\}. \] Hence $C \circ \theta - C$ has
an integrable minorant, and thus $C$ is adjusted to $\theta$ by
Lemma~\ref{l.adjusted}.
\end{proof}

For the proof of Theorem~\ref{t.random-sturman-stark_intro}, the
following variation of Theorem~\ref{t.random_semiuniform} will be
crucial.
\begin{thm}\label{t.complement-to-main}
  In the situation of Theorem~\ref{t.random_semiuniform}, there exist
  $\lambda'<\lambda$ and $k_0\in\N\setminus\{0\}$ such that for all
  $k\geqslant k_0$ there are an adjusted random variable
  $\CC_k:\Omega\to[0,\infty)$ and an ergodic component\foot{An ergodic
    component of $\theta^k$ is a $\theta^k$-invariant set $\Omega_k$
    of positive measure such that $\theta^k_{|\Omega_k}$ is ergodic.}
  $\Omega_k$ of $\theta^k$ with $\Proj(\Omega_k)\geqslant1/k$ such
  that
\begin{equation}\label{eq:complement-to-main}
\Phi_k(\omega,x)
\leqslant
\CC_k(\theta^k\omega)-\CC_k(\omega)+k\lambda'
\quad\text{for $\Proj$-a.e. $\omega\in\Omega_k$ and all $x\in K(\omega)$.}
\end{equation}
The random variables $\CC_k$ can also be chosen to take values in
$(-\infty,0]$. Furthermore
\begin{compactenum}[a)]
\item if $(\Phi_n)_{n\in\N}$ is additive,\foot{That is, (\ref{e.2}) holds with equality.} then $k_0=1$ so that
  (\ref{eq:complement-to-main}) holds for $k=1$ and $\Proj$-a.e.
  $\omega\in\Omega$;
\item if $\theta$ is totally ergodic,\foot{That is, $\theta^k$ is
    ergodic for all $k\in\N$.} then (\ref{eq:complement-to-main})
  holds for $\Proj$-a.e. $\omega\in\Omega$.
\end{compactenum}
\end{thm}

\begin{proof}
  As $\overline\Phi^{\Kmax}<\lambda$ by
  Lemma~\ref{l.maximizing_measure}, there is $k_0\in\N$ such that
  $\E_\Proj[\Phi_k^{\Kmax}]=\int_\Omega \Phi_k^{\Kmax} \ d\Proj <k\lambda'$
  for some $\lambda'<\lambda$ and all $k\geqslant k_0$. Fix any such
  $k$.  Then
\begin{equation}
  \nLim\frac{1}{n}\sum_{j=1}^{n}\Phi_k^{\Kmax} (\theta^{- jk} \omega)
  \ = \ \E_\Proj\left[\Phi_k^{\Kmax}\ |\ \cI_k\right](\omega)\quad\text{for $\Proj$-a.e. $\omega$,}
\end{equation}
where $\E_\Proj[\ . \ |\cI_k]$ denotes the conditional expectation
w.r.t. the $\sigma$-algebra $\cI_k\subseteq\cF$ of all
$\theta^k$-invariant sets. Since $\theta$ is ergodic, all sets of
positive measure in $\cI_k$ have measure at least $1/k$. As
$\int_\Omega\E_\Proj\left[\Phi_k^{\Kmax}\ |\ \cI_k\right]\,d\Proj
=\E_\Proj\left[\Phi_k^{\Kmax}\right]<k\lambda'$, this means that there
is an ergodic component of $\theta^k$ such that
$\nLim\frac{1}{n}\sum_{j=1}^{n}\Phi_k^{\Kmax} (\theta^{- jk} \omega)
<k\lambda'$ for $\Proj$-a.e. $\omega\in\Omega_k$. Hence,
\begin{equation}
		0\ \leqslant\
  \CC_k(\omega) \ = \ \sup_{n \geqslant 0} \left( - \lambda' nk + \sum_{j = 1}^n
    \Phi_k^{\Kmax} (\theta^{- jk} \omega) \right)\ <
    \ \infty \quad\text{for $\Proj$-a.e. $\omega\in\Omega_k$.}
    \label{eq:Domega}
\end{equation}
Let $\omega\in\Omega_k$.  We prove (\ref{eq:complement-to-main}):
  \begin{eqnarray*}    
    \CC_k(\theta^k \omega) & = & \sup_{n \geqslant 0} \left( - \lambda' nk +
      \sum_{j = 0}^{n - 1} \Phi_k^{\Kmax} (\theta^{- jk} \omega) \right)\\
    & \geqslant & \sup_{n \geqslant 1} \left( - \lambda' nk +
      \sum_{j = 0}^{n - 1} \Phi_k^{\Kmax} (\theta^{- jk} \omega) \right)\\
    & = & 
    \sup_{n \geqslant 1} \left( - \lambda' (n - 1) k 
      + \sum_{j = 1}^{n - 1} \Phi_k^{\Kmax} (\theta^{- jk} \omega) \right) - \lambda' k
    + \Phi_k^{\Kmax} (\omega)\\
    & = & 
    \CC_k (\omega) - \lambda' k + \Phi_k^{\Kmax} (\omega) \ .
  \end{eqnarray*}
  This also implies that $\CC_k \circ \theta^k - \CC_k$ has the
  integrable minorant $-\lambda' k+\Phi_k^{\Kmax}$ so that, in view of
  {\cite[Lemma 4.1.13]{Keller1997b}}, $\CC_k \circ \theta^k - \CC_k
  \in L^1_{\Proj}$ and $\int \CC_k \circ \theta^k - \CC _kd\Proj= 0$.
  Thus, Lemma~\ref{l.adjusted} implies that $\CC_k$ is adapted to
  $\theta^k$. In order to show that it is also adapted to $\theta$,
  let $\ell\in\{0,\dots,k-1\}$. Then Birkhoff's Ergodic Theorem
  implies that
  \begin{eqnarray*}
    0
    &\leqslant&
    \nLim\frac{1}{nk+\ell}	\CC_k(\theta^{nk+\ell}\omega)
    \ = \
    \nLim\frac{1}{n	}\sum_{j=0}^{n-1}\frac{1}{k}
  \left(\CC_k\circ\theta^k-\CC_k\right)(\theta^{jk}(\theta^\ell\omega))\\
    &=&
    \frac{1}{k}\E_\Proj\left[\CC_k\circ\theta^k-\CC_k\ | \ \cI_k\right](\theta^\ell\omega)
    \quad\text{for $\Proj$-a.e. $\omega$.}
  \end{eqnarray*}
  As $\int\E_\Proj\left[\CC_k\circ\theta^k-\CC_k\ | \
    \cI_k\right]\,d\Proj =\int\CC_k\circ\theta^k-\CC_k\,d\Proj=0$, it
  follows that the limit is actually equal to zero for $\Proj$-a.e.
  $\omega$. Hence $\CC_k$ is adapted to $\theta$.  The case $n
  \rightarrow - \infty$ is similar.

  
  Observe that the $\CC_k$ are all non-negative. We finally show how to modify
  the above construction to obtain non-positive $\CC_k$.  To this end let
\begin{equation}
  \CC_k(\omega) \ = \ -\sup_{n \geqslant 0} \left( -\lambda' nk + \sum_{j = 0}^{n-1}
    \Phi_k^{\Kmax} (\theta^{jk} \omega) \right) \ .\label{eq:Domega2}
\end{equation}  
As above one shows that there is an ergodic component $\Omega_k$ of $\theta^k$
such that $-\infty<\CC_k(\omega)\leqslant 0$ for
$\Proj$-a.e. $\omega\in\Omega_k$.  We prove (\ref{eq:complement-to-main}):
  \begin{eqnarray*}
    \CC_k(\theta^k \omega) & = & \inf_{n \geqslant 0} \left( \lambda' nk -
      \sum_{j = 1}^{n } \Phi_k^{\Kmax} (\theta^{jk} \omega) \right)\\
    & = & 
    \inf_{n \geqslant 0} \left( \lambda' (n + 1) k 
      - \sum_{j = 0}^{(n+1) - 1} \Phi_k^{\Kmax} (\theta^{jk} \omega) \right) - \lambda' k
    + \Phi_k^{\Kmax} (\omega)\\
    & \geqslant & 
    \CC_k (\omega) - \lambda' k + \Phi_k^{\Kmax} (\omega) \ .
  \end{eqnarray*}
  Again, $\CC_k \circ \theta^k - \CC_k$ has the integrable minorant
  $-\lambda' k+\Phi_k^{\Kmax}$, and as before one proves that $\CC_k$
  is adjusted to $\theta$.
\end{proof}

\begin{rem}
  \label{r.minus_infinity} As mentioned, all the above results can
  also be applied to subadditive sequences $\nfolge{\Phi_n}$ of random
  continuous functions taking values in $\R\cup\{-\infty\}$ (equipped
  with the obvious topology). In order to see this, suppose
  $\overline\Phi_\mu<\lambda$ for all ergodic invariant measures
  $\mu\in\cM^K_{\Proj}(T)$. Then Kingman's Subadditive Ergodic Theorem,
  which still applies since it does not assume continuity of the
  $\Phi_n$, implies that for all ergodic $\mu\in\cM^K_{\Proj}(T)$ we
  have
\[
       \limsup_{n\to\infty} \ntel\Phi_n \ < \ \lambda \text{\; $\mu$-a.e.}
  \]
  Let $\lambda'< \lambda$ and
  $\Phi'_n(\omega,x)=\max\{n\lambda',\Phi_n(\omega,x)\}$. Then
  $\nfolge{\Phi_n'}$ is subadditive and
  \[
       \limsup_{n\to\infty} \ntel\Phi_n' \ < \ \lambda \text{\; $\mu$-a.e. } 
  \]
  and hence $\overline\Phi'_\mu<\lambda$ for all ergodic
  $\mu\in\cM^K_{\Proj}(T)$. Thus, we can apply the above results to
  the subadditive sequence $\nfolge{\Phi_n'}$ of continuous real-valued
  functions. Since $\Phi_n'\geq \Phi_n$, all estimates then carry
  over to the original sequence.
\end{rem}

\section{Continuity of random invariant graphs} \label{Applications} 

We now consider random dynamical systems with a double skew product
structure
\begin{equation}
  \label{e.10}
  T :\Omega\times M \to\Omega\times M \quad , 
  \quad T(\omega,\xi,y) = (\theta\omega,g_\omega(\xi),h_{\omega,\xi}(y)) \ ,
\end{equation}
where $M=\Xi\times\R^d$ as above and $\Xi$ is again a compact metric
space. We assume that the maps $g_\omega$ are homeomorphisms, the maps
$(\xi,y)\mapsto h_{\omega,\xi}(y)$ are continuous and differentiable
in $y$ and $Dh_{\omega,\xi}(y)$ is continuous in $(\xi,y)$ for all
$\omega\in\Omega$. As mentioned, the action of the $g_\omega$ on the
second component $\Xi$ corresponds to the deterministic part of the
forcing, such that the combined forcing process is the {\em random
  homeomorphism}
\begin{equation}
  \theta\ltimes g : \Omega\times \Xi \to \Omega\times\Xi \quad , 
  \quad \theta\ltimes g(\omega,\xi) = (\theta\omega,g_\omega(\xi)) \ . 
\end{equation}
In this way, given any $\theta\ltimes g$-invariant probability measure
$m$ we can view $T=(\theta\ltimes g)\ltimes h$ as a random map over
the base $(\Omega\times\Xi,\cF\times\cX,m,\theta\ltimes g)$, where
$\cX$ denotes the Borel $\sigma$-algebra on $\Xi$. An alternative
point of view is to write $T=\theta\ltimes(g\ltimes h)$, thus
interpreting $T$ as a random map over the base
$(\Omega,\cF,\Proj,\theta)$. In this case the fibre maps $T_\omega$
have a skew product structure themselves, and $T_\omega:M\to M$ is a
continuous transformation of the form
$T_\omega(\xi,y)=(g_\omega(\xi),h_{\omega,\xi}(y))$.  Given $T$ as in
(\ref{e.10}) and a random $T$-invariant set $K$, we let
$K(\omega)=\{(\xi,y)\in M\mid (\omega,\xi,y)\in K\}$ and
$K(\omega,\xi)=\{y\in\R^d\mid (\omega,\xi,y)\in K\}$. Then
$T_\omega(K(\omega))=K(\theta\omega)$ and
$h_{\omega,\xi}(K(\omega,\xi)) = K(\theta\ltimes g(\omega,\xi))$ (see
Lemma~\ref{l.K-twice} below).

Similar to above, the Lyapunov exponent of a $T$-invariant measure
$\mu$ is defined as
\begin{equation}
  \lambda(\mu,T) \ = \ \mu(\bar \Phi) \ = \ \inf_{n\in\N}\ntel \mu(\Phi_n) \ ,
\end{equation}
where now the subadditive sequence $\Phi_n$ is given by
$\Phi_n(\omega,\xi,y)=\log\|D_y h^{n}_{\omega,\xi}(y)\|$. We will
prove the following slightly more general version of
Theorem~\ref{t.random-sturman-stark_intro}.

\begin{thm}\label{t.random-sturman-stark}
  Let $T$ be a random map of the form (\ref{e.10}) and $K$ be a
  $T$-invariant random compact set.  Suppose that for all $k\in\N$ and
  all $\epsilon>0$ there is $r>0$ such that
\begin{equation}\label{eq:extra-uniform}
\begin{split}
  \sup&\left\{\Phi_k(\omega,\xi,y)\mid \omega\in\Omega, (\xi,y)\in
    B_r(K(\omega))\right\}
  \\
  & \leqslant \  \sup\left\{\Phi_k(\omega,\xi,y)\mid \omega\in\Omega,
    (\xi,y)\in K(\omega)\right\} \ + \ \epsilon 
\end{split}
\end{equation}
where $B_r(K(\omega))=\{(\xi,y)\in\Xi\times\R^d\mid
d((\xi,y),K(\omega))<r\}$ and $d$ is a canonical product metric on
$\xi\times\R^d$.  Then, if $\lambda(\mu,T)<0$ for all measures
$\mu\in\cM_\Proj^K(T)$, and if $\theta\ltimes g$ is a random minimal
homeomorphism on $\Omega\times\Xi$, there are a (non-random) integer
$n>0$ and a random variable $c(\omega)>0$ such that, for $\Proj$-a.e.
$\omega\in\Omega$, \romanlist
\item $\# K(\omega,\xi)=n$ for all $\xi\in\Xi$, 
\item the map $\xi\mapsto K(\omega,\xi)$ from $\Xi$ to $\cK(\R^d)$ is continuous, and
\item for all $\xi\in\Xi$, any two different points $y,y'\in
  K(\omega,\xi)$ have distance at least $c(\omega)$.  \listend
\end{thm}

If $T:\Omega\times\R^d\to\Omega\times\R^d$ is a random map with base
$(\Omega,\cF,\Proj,\theta)$ then we can add a trivial component $\Xi=\{\xi_0\}$
and let $g_\omega(\xi_0)=\xi_0$ to apply
Theorem~\ref{t.random-sturman-stark}. As $\theta\ltimes g$ is certainly random
minimal in this case, this immediately yields
\begin{cor}
  \label{c.singlepoints} Suppose
  $T:\Omega\times\R^d\to\Omega\times\R^d$ is a random
  $\mathcal{C}^1$-map with base $(\Omega,\cF,\Proj,\theta)$ and $K\ssq
  \Omega\times \R^d$ is a random compact set such that
  (\ref{eq:extra-uniform}) is satisfied and $\lambda(\mu,T)<0$ for all
  $\mu\in\cM_\Proj^K(T)$. Then there exists an integer $n$ such that
  $\# K(\omega)= n$ for $\Proj$-a.e.\ $\omega\in\Omega$.
\end{cor}

The minimality assumption on $\theta\ltimes g$ in
Theorem~\ref{t.random-sturman-stark} can also be replaced by requiring
connectedness of the sets $K(\omega,\xi)$. We say the random compact
set $K\ssq \Omega\times \Xi\times\R^d$ has {\em connected fibres} if
$K(\omega,\xi)$ is connected for $\Proj$-a.e.\ $\omega\in\Omega$ and all
$\xi\in\Xi$.
\begin{thm}\label{t.connected}
  Let $T$ be a random map of the form (\ref{e.10}) and $K$ be a
  $T$-invariant random compact set with connected fibres.  Suppose
  that for all $k\in\N$ and all $\epsilon>0$ there is $r>0$ such that
  (\ref{eq:extra-uniform}) holds and that $\lambda(\mu,T)<0$ for all
  measures $\mu\in\cM_\Proj^K(T)$. Then for $\Proj$-a.e.\
  $\omega\in\Omega$ the fibre $K(\omega)$ consists of a single
  continuous graph, that is, there is a random continuous function
  $\phi:\Omega\times\Theta\to\R^d$ such that
  $K(\omega)=\{(\xi,\phi(\omega,\xi))\ \mid\ \xi\in \Xi\}$ for
  $\Proj$-a.e.  $\omega\in\Omega$.
\end{thm}

Both theorems are consequences of the following more technical
proposition. We say an open or closed random set $A\ssq \Omega\times
\Xi$ is {\em non-empty}, if $\{\omega\in\Omega\mid A(\omega)\neq
\emptyset\}$ has positive measure. Note that if $A$ is $(\theta\ltimes
g)$-invariant and $\theta$ is ergodic, then non-emptiness of $A$
implies that $\{\omega\in\Omega\mid A(\omega)=\emptyset\}$ has measure
zero.  Similarly, equality of random sets will always be understood as
an equality of the fibres for $\Proj$-a.e.\ $\omega\in\Omega$.
\begin{prop}\label{p.main}
  Let $T$ be a random map of the form (\ref{e.10}) and $K$ be a
  $T$-invariant random compact set.  Suppose that for all $k\in\N$ and
  all $\epsilon>0$ there is $r>0$ such that (\ref{eq:extra-uniform})
  holds.  Then, if $\lambda(\mu,T)<0$ for all measures
  $\mu\in\cM_\Proj^K(T)$, there are a positive integer $n$, a random
  variable $c:\Omega\to(0,\infty)$ and a non-empty open and
  ${(\theta\ltimes g)}$-forward invariant random set $A$ such that,
  for $\Proj$-a.e. $\omega\in\Omega$, \romanlist
\item $\# K(\omega,\xi)=n$ for all $\xi\in A(\omega)$, 
\item\label{it:prop2} $\sup\left\{\# K(\omega,\xi)\mid \xi\in\Xi\right\}<\infty$,
  and
\item\label{it:prop3} for  all $\xi\in\Xi$, any two different
points $y,y'\in K(\omega,\xi)$ have distance at least $c(\omega)$. 
\listend
\end{prop}
\begin{proof}[\bf Proof of Theorem~\ref{t.random-sturman-stark}]
  As $A$ is a nonempty forward $(\theta\ltimes g)$-invariant random
  set, $A^c$ is a backwards $(\theta\ltimes g)$-invariant random
  compact set and $A^c(\omega)\neq\Xi$ for $\Proj$-a.e.\ $\omega$. As
  $\theta\ltimes g$ is minimal, $A^c(\omega)=\emptyset$ for
  $\Proj$-a.e. $\omega$ (see Lemma~\ref{lem:minimal-equiv}(ii) below),
  and hence $A(\omega)=\Xi$ for $\Proj$-a.e. $\omega$.  Assertions (i)
  and (iii) of the proposition together with the compactness of
  $K(\omega)$ imply the continuity of the map $\xi\mapsto
  K(\omega,\xi)$ from $\Xi$ to $\cK(\R^d)$.
\end{proof}
\begin{proof}[\bf Proof of Theorem~\ref{t.connected}]
  By assertion (ii) of the proposition, together with the
  connectedness of the fibres of $K$, there is a subset
  $\Omega_0\subseteq\Omega$ of full measure such that
  $K(\omega,\xi)\subset\R^d$ consists of a single point for all
  $(\omega,\xi)\in\Omega_0\times\Xi$. As
  $K(\omega)\subset\Xi\times\R^d$ is compact, it must be the graph of
  a continuous map $\phi(\omega,\,.\,):\Xi\to\R^d$. As
  $\{\phi(\omega,\xi)\}=K(\omega,\xi)$ is the only possible selection
  of $K$, $\phi$ is measurable \cite[Proposition 1.6.3]{Arnold1998RandomDynamicalSystems}.
\end{proof}

\begin{proof}[\bf Proof of Proposition~\ref{p.main}]
  By Theorem~\ref{t.complement-to-main} there exist $\lambda'<0$,
  $k\in\N\setminus\{0\}$, an ergodic component $\Omega_k$ of $\theta^k$, and an
  adjusted random variable $\CC:\Omega\to(-\infty,0]$ such that
\begin{displaymath}
  \Phi_k(\omega,\xi,y)
  \leqslant
  \CC(\theta^k\omega)-\CC(\omega)+k\lambda'
  \quad\text{for $\Proj$-a.e. $\omega\in\Omega_k$ and all $(\xi,y)\in K(\omega)$.}
\end{displaymath}
Hence, in view of assumption (\ref{eq:extra-uniform}), there are $r>0$ and
$\eta>0$ such that
\begin{equation}\label{eq:fibre-wise_contraction}
  \log\left\|D_y h^{k}_{\omega,\xi}(y)\right\|
  \leqslant
  \CC(\theta^k\omega)-\CC(\omega)- \eta
  \quad\text{for $\Proj$-a.e. $\omega\in\Omega_k$ and all
   $(\xi,y)\in B_{r}\left( K(\omega)\right)$.}
\end{equation}

Now we consider $T$ as $(\theta\ltimes g)\ltimes h$. We defer the
proof of the following two lemmas to the end of this section.
\begin{lemma}\label{l.K-twice}
  $K$ is a random compact set over the base
  $(\Omega\times\Xi,\cF\times\cB)$, and $h_{\omega,\xi}(K(\omega,\xi))=
  K((\theta\ltimes g)(\omega,\xi))$ for $\Proj$-a.e.all $\omega\in\Omega$ and 
  $\xi\in\Xi$.
\end{lemma}
Next, for $\epsilon\in(0,r]$, define $N_\epsilon(\omega,\xi)$ to be the smallest
number of open balls $B_{\epsilon e^{\CC(\omega)}}(y)\subset\R^d$ centred at
points $y\in K(\omega,\xi)$ that are needed to cover the compact set
$K(\omega,\xi)$.
\begin{lemma}\label{l.N_eps}
\begin{enumerate}[a)]
\item The $N_\epsilon$ are $\cF\times\cB$-measurable. 
\item For each $\omega\in\Omega$, the function $N_\epsilon(\omega,\ .\
  ):\Xi\to\N$ is upper semicontinuous.
\end{enumerate}
\end{lemma}

Fix $\epsilon\in(0,r]$, $\omega\in\Omega_k$ and $\xi\in\Xi$, and denote
$N=N_\epsilon(\omega,\xi)$.  There are $y_1,\dots,y_N\in K(\omega,\xi)$ such
that $K(\omega,\xi)\subset\bigcup_{i=1}^N B_{\epsilon e^{\CC(\omega)}}(y_i)$.
As $h_{\omega,\xi}(K(\omega,\xi))= K((\theta\ltimes g)(\omega,\xi))$, it follows
from (\ref{eq:fibre-wise_contraction}) that
\begin{displaymath}
K((\theta\ltimes g)^k(\omega,\xi))
\ \subset \
\bigcup_{i=1}^N h_{\omega,\xi}^k\left( B_{\epsilon e^{\CC(\omega)}}(y_i)\right)
\ \subset \
\bigcup_{i=1}^N B_{e^{-\eta} \epsilon e^{\CC(\theta^k\omega)}}(h_{\omega,\xi}^k( y_i))\quad\text{for $\Proj$-a.e. $\omega\in\Omega_k$}
\end{displaymath}
with points $h_{\omega,\xi}^k( y_i)\in K((\theta\ltimes g)^k(\omega,\xi))$.
Hence 
\begin{equation}\label{eq:N-order}
N_{\epsilon}((\theta\ltimes g)^k(\omega,\xi))
\ \leqslant \
N_{e^{-\eta}\epsilon}((\theta\ltimes g)^k(\omega,\xi))
\ \leqslant \
N_\epsilon(\omega,\xi)\ .
\end{equation}

Consider the restricted system ${(\theta\ltimes
  g)^k}_{|\Omega_k\times\Theta}$ and denote the normalised probability
measure $\Proj_{|\Omega_k}$ by $\Proj_k$.  By Lemma~\ref{l.N_eps},
there is a subset $\Omega_k'\subseteq\Omega_k$ of full measure such
that the random sets $U_{\epsilon,\alpha} = \{(\omega,\xi)\in
\Omega_k\times\Xi\mid N_\epsilon(\omega,\xi)<\alpha\}$ are open for
all $\alpha\in\R$ and $\epsilon=e^{-p\eta}r$ with $p\in\N$. For
measurability purposes we restrict to these countably many values of
$\epsilon$ from now on. Let
\[
n_\epsilon(\omega) \ = \ \min\{\alpha\in\N\mid
U_{\epsilon,\alpha}(\omega) \neq \emptyset \} \ .
\]
The measurability of $n_\epsilon$ follows easily from
Lemma~\ref{lem:TopDyn1}\ref{item:topDyn1d}.  Due to (\ref{eq:N-order})
we have $n_\epsilon(\theta^k\omega)\leqslant n_\epsilon(\omega)$ for
$\Proj_k$-a.e.\ $\omega\in\Omega_k$, and thus the ergodicity of
$(\theta^k,\Proj_k)$ implies that all $n_\epsilon$ are constant
$\Proj_k$-a.e. We denote these constant integers by $n_\epsilon$
again.  By the first inequality of (\ref{eq:N-order}),
$n_\epsilon\leqslant n_{e^{-\eta}\epsilon}$. But the second inequality
of (\ref{eq:N-order}) implies that also
$n_{e^{-\eta}\epsilon}\leqslant n_\epsilon$ for all
$\epsilon\in(0,r]$, so that all $n_\epsilon$ coincide. Denote their
common value by $n$.

Using (\ref{eq:N-order}) again, we see that the random open set
$U_{r,n}$ is $(\theta\ltimes g)^k$-invariant and we have
$U_{r,n}=U_{\epsilon,n}$ for all $\epsilon$. Similarly, for each
integer $m>n$ the set $U_{r,m}$ is a non-empty $(\theta\ltimes
g)^k$-invariant open random set and $U_{r,m}=U_{\epsilon,m}$ for all
$\epsilon$. Since the random compact set $K\cap (\Omega_k\times M)$ is
covered by these countably many random open sets $(U_{r,m})_{m\in\N}$,
we see that the following hold for $\Proj_k$-a.e.\
$\omega\in\Omega_k$:
\begin{compactenum}
\item\label{it:1} $\# K(\omega,\xi)=n$ for all $\xi\in U_{r,n}$;
\item $\sup\left\{\# K(\omega,\xi)\mid \xi\in\Xi\right\}<\infty$;
\item\label{it:3} $d(y,y') \geq c(\omega):=re^{\CC(\omega)}/2>0$ for
  all $\xi\in\Xi$ and any two different points $y,y'\in K(\omega,\xi)$.
\end{compactenum}
Let $A_k=U_{r,n}$. If $k=1$, then $A=A_k$ satisfies the assertions of
the proposition. Otherwise, we let $A=\bigcup_{i=0}^{k-1}
(\theta\ltimes g)^i (A_k)$. Then, as
$\bigcup_{i=0}^{k-1}\theta^k(\Omega_k)=\Omega$ up to a set of
$\Proj$-measure zero, as the $g^i_\omega$ are homeomorphisms, and as
$h_{\omega,\xi}^i(K(\omega,\xi))=K((\theta\ltimes g)^i(\omega,\xi))$,
assertions (\ref{it:1})--(\ref{it:3}) carry over to $\Proj$-a.e.\
$\omega\in\Omega$.
\end{proof}

\begin{rem}\label{r.random-sturman-stark}
  If the second alternative ``$K(\omega)=\emptyset$ for $\Proj$-a.e.\
  $\omega$'' in the definition of minimality
  (Definition~\ref{def:minimality}) is replaced by ``$K(\omega)$ is
  nowhere dense for $\Proj$-a.e.\ $\omega$'', $\theta\ltimes g$ is
  called \emph{transitive}.  Random transitivity seems to be a more
  subtle concept than random minimality; some of its aspects are
  discussed in Section~\ref{RandomTopDyn}. Here we only note a version
  of Theorem~\ref{t.random-sturman-stark} for random transitive
  homeomorphisms:

  {\em\noindent If the random homeomorphism in
    Theorem~\ref{t.random-sturman-stark} is not minimal but only
    transitive, then there are a (non-random) integer $n>0$, a
    random $\theta\ltimes g$-invariant open dense set
    $A(\omega)\subseteq\Xi$ and a random variable $c(\omega)>0$ such
    that, for $\Proj$-a.e. $\omega\in\Omega$, \romanlist
  \item $\# K(\omega,\xi)=n$ for all $\xi\in A(\omega)$,
\item the map $\xi\mapsto K(\omega,\xi)$ from $A(\omega)$ to $\cK(\R^d)$ is
  continuous, and
\item for all $\xi\in\Xi$, any two different points $y,y'\in K(\omega,\xi)$ have
  distance at least $c(\omega)$.
\listend
}
\noindent
As in the proof of Theorem~\ref{t.random-sturman-stark}, the result is
deduced from Proposition~\ref{p.main}. The only difference is that the
sets $A^c(\omega)$ are no longer empty, but nowhere dense, so that the
open sets $A(\omega)$ are dense.
\end{rem}

\begin{proof}[\bf Proof of Lemma~\ref{l.K-twice}]
  For each $(\omega,\xi)\in\Omega\times\Xi$, the set
  $K(\omega,\xi)=\{y\in\R^d:\ (\xi,y)\in K(\omega)\}$ is compact
  because $K(\omega)$ is. Denote the metric on $\Xi$ by $\rho$ and
  define, for each $n>0$, a metric $d_n$ on $\Xi\times\R^d$ by
  $d_n\left((\xi,y),(\xi',y')\right)=\|y'-y\|+n\,\rho(\xi',\xi)$. Each
  $d_n$ generates the product topology on $\Xi\times\R^d$. Further,
  $\omega\mapsto d_n\left((\xi,y),K(\omega)\right)$ is measurable for
  all $(\xi,y)\in\Xi\times\R^d$, and $\xi\mapsto
  d_n\left((\xi,y),K(\omega)\right)$ is continuous for all
  $\omega\in\Omega$ and $y\in\R^d$. It follows that
  $(\omega,\xi)\mapsto d_n\left((\xi,y),K(\omega)\right)$ is
  measurable for all $y\in\R^d$, see \cite[Lemma~1.5.2]{Arnold1998RandomDynamicalSystems}
  or \cite[Lemma~3.14]{CV1977}. Hence also
\begin{displaymath}
(\omega,\xi)\mapsto \sup_n d_n\left((\xi,y),K(\omega)\right)
=
\rho\left(y,K(\omega,\xi)\right)
\end{displaymath}
is measurable so that $K(\omega,\xi)$ is indeed a random compact set.  Finally,
$h_{\omega,\xi}(K(\omega,\xi))= K((\theta\ltimes g)(\omega,\xi))$ follows from
the observation
\begin{eqnarray*}
y\in K(\omega,\xi)
&\Leftrightarrow&
(\xi,y)\in K(\omega)\\
&\Leftrightarrow&
T_\omega(\xi,y)=(g_\omega(\xi),h_{\omega,\xi}(y))\in K(\theta\omega)\\
&\Leftrightarrow&
h_{\omega,\xi}(y)\in K((\theta\ltimes g)(\omega,\xi))
\ =\ K(\theta\omega,g_\omega(\xi))\ .
\end{eqnarray*}
\end{proof}

\begin{proof}[\bf Proof of Lemma~\ref{l.N_eps}]\quad\\
  a) By Lemma~\ref{l.K-twice}, $K$ is a random compact set over the base
  $(\Omega\times\Xi,\cF\times\cB)$ so that there is a sequence $\kfolge{a_k}$ of
  measurable maps $a_k:\Omega\times\Xi\to \R^d$ such that
  $K(\omega,\xi)=\operatorname{closure}\{a_k(\omega,\xi)\mid k\in\N\}$ for all
  $(\omega,\xi)\in\Omega\times\Xi$ \cite[Theorem III.30]{CV1977}, see also
  \cite[Proposition 1.6.3]{Arnold1998RandomDynamicalSystems}.

  Let $\epsilon>0$. For $n\in\N$ denote by $\cL_n$ the family of subsets of $\N$
  with $n$ elements. Then the sets
\begin{displaymath}
  V_n := \bigcup_{L\in\cL_n}\bigcap_{k\in\N}\bigcup_{\ell\in L}
    \left\{(\omega,\xi)\in\Omega\times\Xi\mid  
    \|a_k(\omega,\xi)-a_\ell(\omega,\xi)\|<\epsilon\right\}
\end{displaymath}
are $\cF\times\cB$-measurable, and $N_\epsilon(\omega,\xi)\leqslant n$ if and
only if $(\omega,\xi)\in V_n$. This proves the measurability of $N_\epsilon$.\\
b) Let $\omega\in\Omega$. In order to prove that the function $\xi\mapsto
N_\epsilon(\omega,\xi)$ is upper semicontinuous, it suffices to observe that,
for each $\alpha\in\R$, the set $\{\xi\in\Xi \mid N_\epsilon(\omega,\xi)<\alpha\}$
is open, because $K(\omega)$ is compact.
\end{proof}

\section{Random minimality and random transitivity}\label{RandomTopDyn}

We start by collecting a few elementary, but useful facts.

\begin{lemma}\label{lem:TopDyn1}
\begin{enumerate}[a)]
\item \label{item:topDyn1a} If $K$ is a random compact set, then
  $\operatorname{int}(K)$ (the fibre-wise interior of $K$) is a random
  open set.
\item\label{item:topDyn1b} If $\theta\ltimes g$ is a random
  homeomorphism and $K$ is a random compact set, then $(\theta\ltimes
  g)(K)$ is a random compact set with fibres
  $g_{\theta^{-1}\omega}(K(\theta^{-1}\omega))$.
\item\label{item:topDyn1c} If $K_0,K_1,\dots$ are random compact sets, then
  $\bigcap_{n=0}^\infty K_n$ is a random compact set with fibres
  $\bigcap_{n=0}^\infty K_n(\omega)$.
\item \label{item:topDyn1d} If $A$ is a random open or closed set,
  then $\pi_1(A)=\{\omega\in \Omega\mid A(\omega)\neq \emptyset\}$ is
  $\cF$-measurable.
\end{enumerate}
\end{lemma}
\begin{proof}
  \ref{item:topDyn1a}) See \cite[Chapter 3]{Crauel2002}.\\
  \ref{item:topDyn1b}) There is a sequence $a_0,a_1,a_2,\dots$ of
  measurable selections of $K$ such that
  $K(\omega)=\operatorname{closure}\{a_k(\omega): k\in\N\}$ for all
  $\omega\in\Omega$ \cite[Theorem III.30]{CV1977}.  Then the maps
  $g_{\theta^{-1}.}(a_k(\theta^{-1}\,.\,))$ are measurable maps from
  $\Omega$ to $\Xi$ and
  $g_{\theta^{-1}\omega}(K(\theta^{-1}\omega))=\operatorname{closure}
  \{g_{\theta^{-1}\omega}(a_k(\theta^{-1}\omega))\mid k\in\N\}$, so
  that $(\theta\ltimes g)(K)$ is a random compact set by
  \cite[Theorem III.30]{CV1977} again.\\
  \ref{item:topDyn1c}) This is part of \cite[Proposition III.4]{CV1977}.\\
  \ref{item:topDyn1d}) For closed $A$, $\pi_1(A)$ is measurable by
  \cite[Prop. 1.6.2]{Arnold1998RandomDynamicalSystems}. If $A$ is open, then
  $(\pi_1(A))^c=\{\omega\in\Omega\mid
  A^c(\omega)=\Xi\}=\bigcap_G\{\omega\in\Omega\mid A^c(\omega)\cap
  G\neq\emptyset\}$ where the intersection ranges over a countabale
  base $\{G\}$ of the topology of $\Xi$, and this set is measurable by
  \cite[Prop. 1.6.2]{Arnold1998RandomDynamicalSystems} again.
\end{proof}

\begin{definition}
  Let $\theta\ltimes g$ be a random homeomorphism on
  $\Omega\times\Xi$.
\begin{enumerate}[a)]
\item A random point $\xi$ is a measurable map $\xi:\Omega\to\Xi$.
\item The random point $\xi$ is a limit point of the sequence
  $(\xi_n)_{n\in\N}$ of random points, if for each $\omega\in\Omega$
  the point $\xi(\omega)\in\Xi$ is a limit point of the sequence
  $(\xi_n(\omega))_{n\in\N}$.
\item The forward orbit of the random point $\xi$ is the sequence
  $(\xi_n)_{n\in\N}$ of random points defined by
  $\xi_n(\omega)=g_{\theta^{-n}\omega}^n(\xi(\theta^{-n}\omega))$.
  The backward orbit is defined analogously by replacing $n$ with
  $-n$.
\item The omega limit set of a random point $\xi$ is the random
  compact set $\mho_\xi$ with fibres $\mho_\xi(\omega)
  =\bigcap_{k\in\N}\operatorname{closure}\{\xi_n(\omega)\mid n\geqslant
  k\}$ where $(\xi_n)_{n\in\N}$ is the orbit of $\xi$.
\end{enumerate}
\end{definition}
\begin{rem}
\begin{enumerate}[a)]
\item In view of \cite[Theorem III.9]{CV1977} and
  Lemma~\ref{lem:TopDyn1}, $\mho_\xi$ is indeed a random compact set.
\item $\mho_\xi$ is $(\theta\ltimes g)$-invariant.
\item A random point $\zeta$ is a limit point of the orbit of $\xi$ if
  and only if $\zeta$ is a measurable selection of $\mho_\xi$.
\end{enumerate}
\end{rem}

Recall the definition of a random minimal homeomorphism:
\begin{definition} \label{d.minimality}
  The random homeomorphism $\theta\ltimes g$ on $\Omega\times\Xi$ is
  \emph{minimal}, if each $(\theta\ltimes g)$-forward invariant random
  closed set $K$ obeys the following dichotomy:
  \begin{description}
  \item[either\,] $K(\omega)=\Xi$ for $\Proj$-a.e. $\omega$,
\item[or{\phantom{eith}}\,] $K(\omega)=\emptyset$ for $\Proj$-a.e. $\omega$.
\end{description}
\end{definition}

\begin{lemma}\label{lem:minimal-equiv}
  Let $\theta\ltimes g$ be a random homeomorphism on
  $\Omega\times\Xi$. The following are equivalent:
\begin{enumerate}[(i)]
\item $\theta\ltimes g$ is minimal.
\item The dichotomy in Definition~\ref{d.minimality} holds for each
  $(\theta\ltimes g)$-backwards invariant random closed set $K$.
\item The dichotomy in Definition~\ref{d.minimality} holds for each
  $(\theta\ltimes g)$-invariant random closed set $K$.
\item For each random point $\xi$ holds: $\mho_\xi(\omega)=\Xi$ for $\Proj$-a.e. $\omega$.
\end{enumerate}
\end{lemma}

\begin{proof}
  ''(i) $\Rightarrow$ (iii)'' and ''(ii) $\Rightarrow$ (iii)'' are trivial.\\
  (iii) $\Rightarrow$ (iv): For each random point $\xi$, $\mho_\xi$
  is a non-empty $(\theta\ltimes g)$-invariant random compact set.\\
  (iv) $\Rightarrow$ (i): Let $K$ be a $(\theta\ltimes g)$-forward
  invariant non-empty random closed set. It has at least one
  measurable selection $\xi$ \cite[Theorem III.9]{CV1977}. Denote its
  orbit by $(\xi_n)_{n\in\N}$. Then all $\xi_n$ are measurable
  selections of $K$ as well since
  $\xi_n(\omega)=g_{\theta^{-n}\omega}^n(\xi(\theta^{-n}\omega))\in
  g_{\theta^{-n}\omega}^n(K(\theta^{-n}\omega))\subseteq K(\omega)$.
  Hence,
  $\Xi=\mho_\xi(\omega)\subseteq K(\omega)$ for $\Proj$-a.e. $\omega$.\\
  (iii) $\Rightarrow$ (ii): Let $K$ be a $(\theta\ltimes g)$-backward
  invariant random closed set.  For each $\omega$ let
  $K'(\omega)=\bigcap_{n=0}^\infty
  (g_{\omega}^n)^{-1}(K(\theta^{n}\omega))$. As a decreasing
  intersection of non-empty random compact sets $K'$ is a non-empty
  random compact set, see Lemma~\ref{lem:TopDyn1}.  It is invariant,
  because
\begin{displaymath}
g_{\omega}(K'(\omega))
\ =\
g_\omega\left(\bigcap_{n=1}^\infty (g_{\omega}^n)^{-1}(K(\theta^{n}\omega))\right)
\ =\
\bigcap_{n=0}^\infty (g_{\theta\omega}^{n})^{-1}(K(\theta^{n}\theta\omega))
\ =\
K'(\theta\omega)\ .
\end{displaymath}
Hence $K'(\omega)=\Xi$ for $\Proj$-a.e. $\omega$. As
$K'(\omega)\subseteq K(\omega)$, the same holds for $K(\omega)$.
\end{proof}

It is tempting to characterise random transitive sets in the same way.
We will see, however, that the situation is more complicated.
\begin{definition} \label{d.transitivity}
  The random homeomorphism $\theta\ltimes g$ on $\Omega\times\Xi$ is
  \emph{transitive}, if each $(\theta\ltimes g)$-forward invariant
  random closed set $K$ obeys the following dichotomy:
\begin{description}
\item[either\,] $K(\omega)=\Xi$ for $\Proj$-a.e. $\omega$,
\item[or{\phantom{eith}}\,] $K(\omega)$ is nowhere dense for $\Proj$-a.e. $\omega$.
\end{description}
\end{definition}

In the same way as for Lemma~\ref{lem:minimal-equiv}, one proves
\begin{lemma}\label{lem:transitive-equiv}
  Let $\theta\ltimes g$ be a random homeomorphism on
  $\Omega\times\Xi$. The following are equivalent:
\begin{enumerate}[(i)]
\item $\theta\ltimes g$ is transitive.
\item The dichotomy in Definition~\ref{d.transitivity} holds for each
  $(\theta\ltimes g)$-backwards invariant random closed set $K$.
\item The dichotomy in Definition~\ref{d.transitivity} holds for each
  $(\theta\ltimes g)$-invariant random closed set $K$.
\end{enumerate}
\end{lemma}

Recall that for a homeomorphisms $g$ of a compact metric space $\Xi$,
the following two statements are equivalent to topological
transitivity: (1) Given any open sets $U,V\ssq \Xi$ there exists
$n\in\N$ with $g^n(U)\cap V\neq \emptyset$ and (2) there exists a
point $x\in \Xi$ with dense orbit.

The first of these equivalent characterisations can easily be carried
over to random dynamical systems. Recall that we say a random set $R$
(open or compact) is {\em non-empty} if $\{\omega\in\Omega\mid
R(\omega) \neq \emptyset\}$ has positive measure.
\begin{lem} \label{l.random-transitivity-one} A random homeomorphism
  $\theta\ltimes g$ on $\Omega\times\Xi$ is \emph{transitive} if and
  only if
\begin{itemize}
\item[\bf(T1)] for all non-empty random open sets $U,V\ssq
  \Omega\times\Xi$ there exists $n\in\N$ such that $(\theta\ltimes
  g)^n(U)\cap V$ is non-empty.
\end{itemize}
\end{lem}
\begin{proof} Suppose $\theta\ltimes g$ is transitive and $U,V\ssq
  \Xi$ are non-empty random open sets. Assume for a contradiction that
\[
    \Proj(\pi_1((\theta\ltimes g)^n(U)\cap V)) \ = \ 0
\]
for all $n\in\N$. Let $\cU = \ncup (\theta\ltimes g)^n(U)$. Then $\cU$
is forward invariant and the set
\[
\Omega' \ = \ \{\omega\in\Omega \mid \cU(\omega)\cap V(\omega)
=\emptyset \} \ = \ \ncup \pi_1((\theta\ltimes g)^n(U)\cap V)
\]
has measure zero. Hence, the backwards $(\theta\ltimes g)$-invariant random compact set
$K=(\Omega\times\Xi)\smin\cU$ contains $V'=V\smin (\Omega'\times\Xi)$, but
is disjoint from $U$. By Lemma~\ref{lem:transitive-equiv}(ii), this
contradicts the transitivity of $\theta\ltimes g$.

Conversely, suppose that (T1) holds and $K$ is a forward
$(\theta\ltimes g)$-invariant compact set such that
$\inte(K(\omega))\neq \emptyset$ $\Proj$-a.s. Let $V=(\Omega\times
\Xi)\smin K$ and assume for a contradiction that $V$ is non-empty, such
that we do not have $K(\omega)=\Xi$ for \Proj-a.e.\ $\omega\in\Omega$.
Define the random open set $U$ by
$U(\omega)=\inte(K(\omega))$, see Lemma~\ref{lem:TopDyn1}\ref{item:topDyn1a}.  Then, by
contradiction assumption, $U\ssq K$ is non-empty, and due to the
forward $(\theta\ltimes g)$-invariance of $K$ we have $(\theta\ltimes
g)^n(U)\cap V\ssq K\cap V = \emptyset$ for all $n\in\N$. This
contradicts (T1).
\end{proof}

In order to generalise (2) to the random situation, one might hope to
characterise random transitivity by the existence of a random point
$\xi$ such that $\mho_\xi(\omega)=\Xi$ for $\Proj$-a.e.\  $\omega$. The
following simple example shows that this cannot work: Let $\Omega=\Xi$
be a compact metric space and consider the random homeomorphism
$\theta\ltimes g$ on $\Omega\times\Xi$ with a minimal homeomorphism
$\theta:\Omega\to\Omega$, an ergodic probability $\Proj$ on $\Omega$
with full topological support, and $g_\omega(\xi)=\xi$ for all
$(\omega,\xi)$. Except in the case of trivial $\Xi$, this is clearly
not random transitive, because $(\Omega\times G)$ is an open
$(\theta\ltimes g)$-invariant set for each open $G\subset \Xi$. On the
other hand, $\xi(\omega)=\omega$ defines a random point with
$\xi_n(\omega)=\xi(\theta^{-n}\omega)=\theta^{-n}\omega$ for all
$\omega$ and $n$. Then, by minimality of $\theta$, we have
$\mho_\xi(\omega)=\Xi$ for all $\omega\in\Omega$.

Hence, in order to make sense of (2) in a random setting we need a
refined concept of `dense orbit'. Given a random point $\xi$ and a set
$A\ssq \Omega$ of positive measure, we call the restriction of $\xi$
to $A$, denoted by $\xi^{A} : A \to \Xi,\ \omega\mapsto \xi(\theta)$,
a {\em subsection} of $\xi$.  We define $\xi^A_n : \theta^n(A)\to \Xi$
by $\xi^A_n(\omega) = g_{\theta^{-n}\omega}^n(\xi^A(\theta^{-n}\omega))$
and define the omega limit $\mho_{\xi^A}$ of $\xi^A$ fibre-wise as
$\mho_{\xi^A}(\omega) \ = \ \bigcap_{k\in\N}\operatorname{closure}\{
\xi^A_n(\omega)\mid n\geqslant k \text{ s.t. } \omega\in\theta^n(A)\}$.

\begin{rem}The following observations are easy to prove.
\begin{enumerate}[a)]
\item $\mho_{\xi^A}$ is a random compact set.
\item $\mho_{\xi^A}$ is $(\theta\ltimes g)$-invariant.
\item Item (iv) in Lemma~\ref{lem:minimal-equiv} can be replaced by 
  \begin{itemize}
  \item[(iv)'] $\mho_{\xi^A}(\omega)=\Xi$ $\Proj$-a.s.\ for every
    subsection $\xi^A$ of any random point $\xi$.
  \end{itemize}
\end{enumerate}
\end{rem}

Using this concept, we have the following.
\begin{lemma} \label{random-transitive-two} Suppose for the random
  homeomorphism $\theta\ltimes g$ there exists a random point $\xi$
  such that $\mho_{\xi^A}(\omega)=\Xi$ $\Proj$-a.s.\ for all
  subsections $\xi^A$ of $\xi$. Then $\theta\ltimes g$ is transitive.
\end{lemma}
\begin{proof} Let $\xi$ be the random point with the above property.
  Suppose $K$ is a $(\theta\ltimes g)$-invariant random compact set
  such that $\{\omega\in\Omega\mid \inte(K(\omega)) \neq \emptyset\}$
  has positive measure. Let $U\ssq K$ be the non-empty random open set
  defined by $U(\omega)=\inte(K(\omega))$.

  As $\mho_\xi(\omega)=\Xi$ for $\Proj$-a.e.\ $\omega\in\Omega$, the
  set
 \[
 \{ \omega\in\Omega\mid \exists n\in\N: \xi_n(\omega)\in U(\omega)\} \
 = \ \ncup\{ \omega\in\Omega\mid \xi_n(\omega)\in U(\omega)\}
 \]
 has full measure. Hence, there exists $n\in\N$ such that
 $A'=\{\omega\in\Omega\mid \xi_n(\omega)\in U(\omega)\}$ has positive
 measure. Then $A=\theta^{-n}(A')$ satisfies $\xi^A(\omega)\in
 K(\omega)$ for all $\omega\in A$, and consequently $\mho_{\xi^A}\ssq
 K$. By assumption, this implies $K(\omega)=\Xi$ \Proj-a.s. Since $K$
 was an arbitrary $(\theta\ltimes g)$-invariant random compact set,
 Lemma~\ref{lem:transitive-equiv}(iii) implies that $\theta\ltimes g$
 is transitive.
\end{proof}
This suggests that a natural way to generalise (2) would be to show
that random transitivity is equivalent to the existence of a random
point $\xi$ whose subsections all have dense orbit (in the sense that
$\mho_{\xi^A}=\Xi$ for $\Proj$-a.e.\ $\omega\in\Omega$). However, we
are at the moment not able to prove this, nor to give a
counterexample. A positive result may depend on additional assumptions
on the base $(\Omega,\mathcal{B},\Proj,\theta)$. We thus have to leave
the following problem open.
\begin{problem}
  Show that under suitable assumptions on
  $(\Omega,\mathcal{B},\Proj,\theta)$ a transitive random
  homeomorphism $\theta\ltimes g$ has a random point $\xi$ whose
  subsections all have dense orbit.  
\end{problem}

\begin{example}
A natural class of examples of
minimal random homeomorphisms is given by random circle rotations,
although it needs a little work to give a rigorous proof of their
minimality. We provide a sketch of the
argument.  Given a measure-preserving dynamical system
$(\Omega,\cB,\Proj,\theta)$ and a measurable function
$\alpha:\Omega\to\Seins=\{z\in\C\mid |z|=1\}$, we let $\Xi=\Seins$ and
define a random homeomorphism $\theta\ltimes g^\tau$ with parameter
$\tau\in\Seins$ by
\begin{equation} \label{e.minimalrotations}
   g^\tau_\omega (z) \ = \  \alpha(\omega)\tau z \ .
\end{equation}
It is easy to see that $\Proj\times\Leb_{\Seins}$ belongs to
$\cM_\Proj(\theta\ltimes g^\tau)$. Further, using the random Krylov-Bogolyubov
Procedure \cite[Theorem 1.5.8]{Arnold1998RandomDynamicalSystems} it is possible to show that if
$\Proj\times\Leb_{\Seins}$ is the unique measure in $\cM_\Proj(\theta\ltimes
g^\tau)$, then $\theta\ltimes g$ is minimal.  An old result of Furstenberg
\cite[Lemma 2.1]{furstenberg:1961} states that the existence of further measures
in $\cM_\Proj(\theta\ltimes g^\tau)$ implies that the Koopman operator of
$\theta$ has a unimodular measurable generalised eigenfunction $R$ in the sense
that
\[
R(\theta\omega)=(\tau\alpha(\omega))^n\ R(\omega)\quad\text{for some $n\in\Z$.}
\]
Note that while the result in \cite{furstenberg:1961} is stated for continuous
and uniquely ergodic base transformations $\theta$, the proof also works in the
general case. If, for a given $n\in\Z$, there are two numbers $\tau$ and
$\tilde{\tau}$ for which this identity holds (with eigenfunctions $R$ and
$\tilde{R}$), then $\tilde{R}R^{-1}$ is an eigenfunction of the Koopman operator
with eigenvalue $(\tilde{\tau}\tau^{-1})^n$.  Now, if $\theta$ is weakly mixing,
then the only $L^2$-eigenvalue of the associated Koopman operator is $1$, so
$\tilde{\tau}\tau^{-1}$ is a root of unity. As a consequence, there exist only
countably many $\tau\in\Seins$ such that $\cM_\Proj(\theta\ltimes g^\tau)$ contains
more than one measure. Hence, $g^\tau$ is minimal for all but countably many
$\tau\in\Seins$.

Basically the same argument works when $\theta$ is a minimal rotation
of a $d$-dimensional torus. In this case, it follows easily from
Fourier analysis that any iterate of $\theta$ can only have countably
many $L^2$-eigenvalues, such that again $g^\tau$ is minimal for all
but countably many $\tau\in\Seins$.
\end{example}


\end{document}